\documentclass[reqno]{amsart}

\usepackage[english]{babel}
\usepackage[utf8]{inputenc}
\usepackage{enumerate}
\usepackage{graphicx}

\usepackage[colorlinks=false,backref=false,pagebackref=false,pdftex,pdfauthor={Vanderley Alves Ferreira Jr, Ederson Moreira dos Santos},pdftitle={On the finite space blow up of the solutions of the Swift-Hohenberg equation}]{hyperref}

\usepackage{amsthm}
\usepackage{amsmath,amsfonts,amssymb}

\usepackage{latexsym,hyperref}
\usepackage{comment}

\usepackage[showonlyrefs]{mathtools}
\mathtoolsset{showonlyrefs=true}

\newcommand{\R}{{\mathbb R}}

\newcommand{\N}{{\mathbb N}}

\newcommand {\menos}{\backslash}
\newcommand {\gd}{\displaystyle}

\newcommand {\rt}{\rightarrow}

\numberwithin{equation}{section}

\newtheorem{theorem}{Theorem}[section]

\newtheorem{lemma}[theorem]{Lemma}
\newtheorem{proposition}[theorem]{Proposition}

\theoremstyle{definition}

\newtheorem{remark}[theorem]{Remark}
\newtheorem*{remarkk}{Remark}

\usepackage{xcolor}

\title{
}

\newtheorem*{namedtheorem}{\theoremname}
  \newcommand{\theoremname}{testing}
  \newenvironment{named}[1]{
     \renewcommand{\theoremname}{#1}
     \begin{namedtheorem}}
     {\end{namedtheorem}}

\addto\captionsenglish{}

\newcommand{\parameter}{\mu}

\begin{document}
 
\title[Blow up of the solutions of the Swift-Hohenberg equation]{On the finite space blow up of the solutions of the Swift-Hohenberg equation}

\thanks{V. Ferreira Jr is  supported by FAPESP \#2012/23741-3 grant. E. Moreira dos Santos is partially supported by CNPq \#309291/2012-7 grant and FAPESP \#2014/03805-2 grant.}

\author{Vanderley Ferreira Jr}
\address{Vanderley A. Ferreira Junior \newline \indent Instituto de Ciências Matemáticas e de Computação \newline \indent Universidade de São Paulo \newline \indent
Caixa Postal 668, CEP 13560-970 - S\~ao Carlos - SP - Brazil}
\email{vanderley.cn@gmail.com}

\author{Ederson Moreira dos Santos}
\address{Ederson Moreira dos Santos \newline \indent Instituto de Ciências Matemáticas e de Computação \newline \indent Universidade de São Paulo \newline \indent
Caixa Postal 668, CEP 13560-970 - S\~ao Carlos - SP - Brazil}
\email{ederson@icmc.usp.br}

\date{\today}


\begin{abstract}
The aim of this paper is to study the finite space blow up of the solutions for a class of fourth order differential equations. Our results answer a conjecture in [F. Gazzola and R. Pavani. Wide oscillation finite time blow up for solutions to nonlinear fourth order differential equations. Arch. Ration. Mech. Anal., 207(2):717–752, 2013] and they have implications on the nonexistence of beam oscillation given by traveling wave profile at low speed propagation.
\end{abstract}
\maketitle

\section{Introduction}

In this paper we consider the equation
\begin{equation} \label{edo1}
w''''(s) + k w''(s)+ f(w(s)) = 0, \quad s \in \mathbb{R},
\end{equation}
where $k \in \mathbb{R}$ and $f: \R \to \R$ is a locally Lipschitz function. When $k$ is positive \eqref{edo1} is referred to as the stationary 1-D Swift-Hohenberg equation whereas when $k$ is negative it is usually called the stationary extended Fisher-Kolmogorov equation (eFK equation). There is a large and diverse literature on the equation \eqref{edo1} and it shows to be a difficult task to mention all of them.  However, we mention the monographs by Collet and Eckmann \cite{colleteckmann}, Cross and Hohenberg \cite{crosshohenberg1993} and Peletier and Troy \cite{peletiertroy} and the papers \cite{swifthohenberg1977, MckennaWalter1987, MckennaWalter1990, lazermckenna, ChenMckenna1997, peletiertroy2000, bergpeletiertroy2001, gazzolapavani2011, gazzolapavani2012, gazzolapavani2013, gazzolabridge} that present many different applications of \eqref{edo1} and that consider nonlinearities $f$ satisfying the same hypotheses as in this paper. 

In particular, the study of \eqref{edo1} with $k>0$ shows to be important in applications. For example, with $k = c^2$, the existence of a solution to \eqref{edo1} corresponds to the existence of a traveling wave solution $u(x,t) = w(x-ct)$ of the nonlinear beam equation
\begin{equation}\label{eqtempo}
u_{tt} + u_{xxxx} + f(u) = 0,  \quad x \in (0,L), \quad t \in \R.
\end{equation}

Throughout this paper we will assume that 
\begin{equation}\label{f1}
f \in {\rm{Lip}}_{{\rm{loc}}}(\R), \quad f(t)t> 0, \ \forall \ t\neq 0,
\end{equation}
and we stress that all of our results apply to the prototype equation
\begin{equation}\label{edo} 
\left\{
\begin{array}{l}
w'''' (s)+ k w''(s)+ \alpha |w(s)|^{q-1}w(s)+|w(s)|^{p-1} w(s) = 0, \quad s \in \R, \\ 
p > q\geq 1,  \ \alpha \geq 0, \ k>0.
\end{array}
\right.
\end{equation}

Let us describe the motivation to write this paper. Gazzola and Pavani \cite{gazzolapavani2013} obtained a careful description of the oscillation of the solutions of the eFK equation, that is \eqref{edo1} in the case of $k \leq0$, and they proved the following theorem.

\begin{named}{Theorem A (\cite[Theorem 2]{gazzolapavani2013})}
\label{thGP2013}
Let $k\leq 0$ and assume that $f$ satisfies \eqref{f1},
\begin{equation}\label{f12}
f \in \mathcal{C}^2(\mathbb{R}\backslash\{0\}), \quad  f''(t)t>0, \ \forall \ t>0,  \quad \liminf_{t\to\pm\infty}|f''(t)|>0,
\end{equation}
and
\begin{align}
&\exists \ p > q \geq 1 , \ \alpha \geq 0, \ 0 < \rho \leq \beta, \ s.t.\notag \\ \
& \rho |t|^{p+1} \leq f(t)t\leq \alpha  |t|^{q+1}+\beta  |t|^{p+1}, \ \quad \forall \ t\in\mathbb{R} \label{f2int}.
\end{align}
Let $w = w(s)$ be a local solution of \eqref{edo1} in a neighborhood of $s =0$ and defined on the maximal interval on the right $[0,R)$. If the initial data satisfy
\begin{equation}\label{dados iniciais GP}
w'(0)w''(0)-w(0)w'''(0)-kw(0)w'(0)> 0,
\end{equation}
then $R<+\infty$ and
\[ \limsup_{s \to R} w(s) = + \infty, \quad  \liminf_{s \to R} w(s) = - \infty .\]
  \end{named}

Then, in the same paper, the authors conjectured that the same result should hold for the Swift-Hohenberg equation; cf. Conjecture B below. However, as pointed out in \cite[p. 728]{gazzolapavani2013}, the arguments used to prove Theorem A, which are based on certain auxiliary functions $G$ and $H$ as well as on some nice properties of the solution $w$, namely \cite[Lemmas 10 and 11]{gazzolapavani2013}, do not apply to prove such conjecture. In this direction let us also quote \cite[p. 721]{gazzolapavani2013}:

\begin{em}
\begin{flushleft}
``It would be interesting to have a similar statement {\rm{(similar to Theorem A)}} when $k>0$, since this would allow us to prove Conjecture 4 {\rm{(Conjecture B below)}}. However, if $k > 0$, there are a couple of important tools which are missing and the proof of Theorem A cannot be extended in a simple way. In any case, numerical results suggest that a result similar to Theorem A also holds for $k>0$.''
\end{flushleft}
\end{em}

\begin{named}{Conjecture B (\cite[Conjecture 4]{gazzolapavani2013})}\label{conjecture4}
Theorem A should also hold for $k>0$.
\end{named}

According to the above comments, many difficulties arise as one tries to solve Conjecture B. Here we undertake this task. To accomplish that we replace the auxiliary functions $G$ and $H$ conveniently and we present a threshold $k_f$ for the blow up.
We prove, by using different arguments, that our functions $G$ and $H$ and the solution $w$ have the same nice monotonicity properties as the corresponding functions in \cite{gazzolapavani2013}, provided that $0 < k \leq k_f$. Going further on this direction, we stress that the wild behavior of the solution $w$ of \eqref{edo} with $f(t) = t^3+ t$,  initial data $(0.8,0,0,0)$ and $k =3.5$ observed in \cite[Fig. 3]{gazzolapavani2013} occurs because $3.5$ is above the corresponding threshold $k_f=2$. For example, in the case  $0 < k \leq 2$, the solution of the same initial value problem has a nicer behavior and the absolute value of the critical values of $w$ are monotone increasing; cf. \eqref{l3} and Fig. \ref{desenho} below. 

\begin{figure}[h]\label{desenho}
\includegraphics{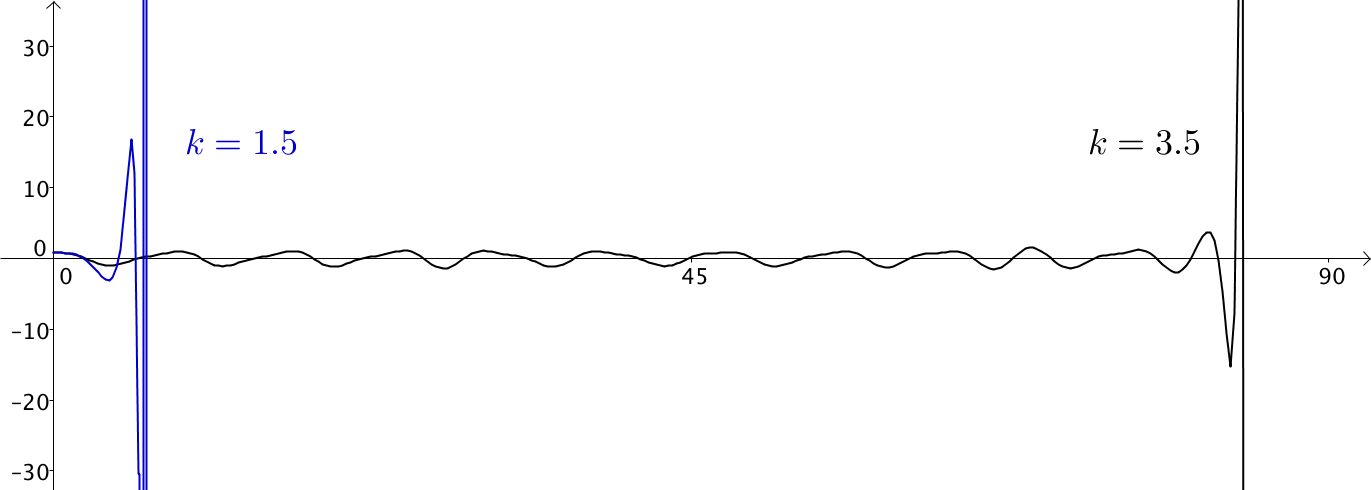}
\vspace{-.3cm}
\caption{Plot of solutions for $k=1.5$ and $k=3.5$.}
\end{figure}
\noindent Note that the initial data in this case satisfies $H(0)=0$, with $H$ as in this paper. Indeed, even in the case that $f$ is not odd symmetric, inequality \eqref{l3} says that the sequences of maxima and minima of $w$ are strictly monotone (assuming all of the hypotheses in Theorem \ref{teoremaprincipal}). In addition, we believe that the threshold $k_f$ might have some intrinsic physical meaning and that our results might contribute towards the understanding of suspension bridges oscillation phenomena. See our comments below about what we call here critical speed for traveling wave solutions of \eqref{eqtempo}.

It seems that the statements in \cite[p. 721 and 728]{gazzolapavani2013} corroborate the impression of the authors of \cite{smetsvanderberg} about the Swift-Hohenberg equation:
\begin{em}
\begin{flushleft}
``In some sense, the Fisher–Kolmogorov situation is simpler to deal with... Concerning the eFK equation, the situation is rather deeply understood... Much less rigorous results exist for the Swift–Hohenberg case.''
\end{flushleft}
\end{em}

Latterly Gazzola and Karageorgis \cite{GazzolaKarageorgis} presented an improvement on the results of Theorem A allowing more general nonlinearities $f$ and initial data possibly violating \eqref{dados iniciais GP}.

\begin{named}{Theorem C (\cite[Theorem 3]{GazzolaKarageorgis})}\label{thGK}
Let $k < 0$. Assume that $f: \R \to \R$ is increasing, $f \in C^1(\R \menos \{0\})$, satisfies \eqref{f1},
\begin{equation}\label{h1GK}
\exists \ c, \delta, \tau > 0, \quad s.t. \quad t f(t) \geq c |t|^{2 + \delta} \quad \forall \ t \in \R, \quad t f(t) \geq c F(t) \quad \forall \ |t|> \tau,
\end{equation}
and
\begin{equation}\label{h2GK}
\exists \ \lambda \in (0,1), \quad \exists \ \alpha > 0, \quad s.t. \quad \gd{\liminf_{t \rt \pm \infty} \dfrac{F (\lambda t)}{F(t)^{\alpha}}> 0}.
\end{equation}
Let $w = w(s)$ be a local solution of \eqref{edo1} in a neighborhood of $s =0$ and defined on the maximal interval on the right $[0,R)$. If the initial data satisfy \eqref{dados iniciais GP} or 
\begin{equation}\label{dados iniciais GK}
\frac{1}{2} (w''(0))^2 - \frac{k}{2} (w'(0))^2 - w'(0)w'''(0) - F(w(0)) \neq 0
\end{equation}
then $R<+\infty$ and
\[ \limsup_{s \to R} w(s) = + \infty, \quad  \liminf_{s \to R} w(s) = - \infty .\]
\end{named}

\begin{remarkk}
As pointed out in \cite{GazzolaKarageorgis}, Theorem C also holds for $k=0$ if $f$ satisfies $t f(t) \geq C t^2$ for all $t \in \R$, where $C$ is a positive constant.
\end{remarkk}

The contribution of this paper is to prove that Conjecture B holds true if and only if $k$ is less or equal to an explicit positive threshold.

We will assume that
\begin{equation}\label{hextra1}
f \in C^1(\R), \quad f'(t) > f'(0)\geq 0, \ \forall  \,  t \neq 0,
\end{equation}
and in some of our results we will suppose
\begin{equation}\label{hextra2}
f'(0)>0.
\end{equation}
Observe that if \eqref{f1} and \eqref{f12} are satisfied and if $f'(0)$ exists, then \eqref{hextra1} holds. 

We set
\[
k_f = 2 \inf_{t \in \mathbb{R}} \sqrt{f'(t)}.
\]
Hence, if \eqref{hextra1} is satisfied then $k_f= 2 \sqrt{f'(0)}$ and for the model problem \eqref{edo}, in the particular case with $f(t) = |t|^{p-1}t+ t$, we have $k_{f}= 2$ (we will comment later on this particular threshold). 

Associated to a solution $w$ of \eqref{edo1} we introduce the function 
\[
H(s)  =  w''(s)w'''(s)+ k w'(s)w''(s)+ f(w(s)) w'(s)
\]
and we state the main result in this paper.

\begin{theorem}\label{teoremaprincipal}
Assume that $f$ satisfies \eqref{f1}, \eqref{f2int}, \eqref{hextra1}, \eqref{hextra2} and let $0<k\leq k_f$.
Let $w$ be a nontrivial solution of (\ref{edo1}) defined on a neighborhood of $s=0$. Let $(R^-, R^+)$ be the maximal interval of existence of $w$. 
\begin{enumerate}[i)]
\item If the initial data satisfy $H(0) \geq 0$, then $R^+<+\infty$,
\[ \limsup_{s \to R^+} w(s) = + \infty \quad \text{and} \quad \liminf_{s \to R^+} w(s) = - \infty .\]
\item If the initial data satisfy $H(0)\leq 0$, then $R^- > - \infty$,
\[ \limsup_{s \to R^-} w(s) = + \infty \quad \text{and} \quad  \liminf_{s \to R^-} w(s) = - \infty .\]
\end{enumerate}
\end{theorem}

For the sake of completeness we also discuss about the existence of periodic solutions of \eqref{edo1} in the case that $k$ is beyond $k_f$.

\begin{theorem}\label{periodics}
Let $f: \R \to \R$ be an odd function satisfying \eqref{f1}, \eqref{f2int} and \eqref{hextra1} {\rm{(}}we emphasize that \eqref{hextra2} is not assumed{\rm{)}}. If $k>k_f$, then there exists $a>0$ and a periodic solution of \eqref{edo1} such that $w(0) = w''(0) =0$, $w'(0)=a$ and $w'''(0) = - k a/2$.
\end{theorem}

\begin{remark}
We have shown, by means of Theorems \ref{teoremaprincipal} and \ref{periodics}, that $k_f$ is the precise threshold for the finite space blow up of solutions of the equation \eqref{edo1}. Let us summarize some consequences of our theorems:

\begin{enumerate} [\textbullet]
\item The two results in Theorem \ref{teoremaprincipal} hold true if we replace the conditions on the sign of $H(0)$ by the respective conditions on $H(s_0)$ at any $s_0 \in (R^-, R^+)$. Indeed, take into account that $w_{s_0}(s) = w(s+ s_0)$ solves \eqref{edo1} provided $w$ does. 
\item Let $0 < k \leq k_f$. Then, by Theorem \ref{teoremaprincipal}, any nontrivial solution of \eqref{edo1} blows up in finite space either to the right or to the left of $s=0$. In this sense Theorem \ref{teoremaprincipal} is stronger than Theorems E and F below on the nonexistence of nontrivial solution of \eqref{edo1} that are globally defined on $\R$.
\item Let $0 < k \leq k_f$ and let $w$ be a nontrivial solution of \eqref{edo1} defined on a neighborhood of $s=0$. According to Theorem \ref{teoremaprincipal}, if $w$ satisfies $H(s_0)=0$ for some $s_0$, then it blows up at finite space to the right and to left of $s=0$. 
\item The solution $w$ of Theorem \ref{periodics} satisfies $H(0)=0$.

\item For any $k > k_f$, the periodic solution $w$ of Theorem \ref{periodics} produces a counter-example for the Conjecture B. Indeed, let $\bar{H}$ be as in \cite{gazzolapavani2013}, namely
\[
\bar{H}(s) = w'(s)w''(s) - w(s)w'''(s) - k w(s)w'(s).
\]
Then the periodic solution of Theorem \ref{periodics} is such that \[\bar{H}(0)=0, \quad \bar{H}'(0)=-k (w'(0))^2<0. \]
Then $\bar{H}$ is negative in a small interval to the right of $0$. Therefore, for some $\varepsilon_0>0$ sufficiently small
\[
w_*(s) = w(-s + \varepsilon_{0})
\]
is a periodic solution of \eqref{edo1} such that $\bar{H}(0)>0$.  
\item Observe that $k_f = 0$ if $f$ verifies \eqref{f2int} with $p> q > 1$ and that $k_f =2$ if $f(t) = |t|^{p-2}t+ t$ with $p>1$.
\item A physical interpretation of Theorem \ref{teoremaprincipal} is that there exists a critical speed, namely $c_* = \sqrt{2{\sqrt{f'(0)}}}$, for the existence of a traveling wave solutions of \eqref{eqtempo}.  No beam oscillation can have a traveling wave profile at low speed propagation, namely less or equal to $c_*$. This result answers to some open questions related to those in \cite[p. 3999]{lazermckenna2011}.
\end{enumerate}\end{remark}

\begin{remark}
The classical stationary 1-D Swift-Hohenberg equation is written as
\begin{equation}\label{SHequation}
\left( 1 + \partial_s^2 \right)^2 U + U^3 -\parameter U = 0 \quad s \in \R, \ \ \text{where} \ \  \parameter \in \R \ \ \text{is a parameter},
\end{equation}
cf. \cite[eq. (3.27)]{crosshohenberg1993}, \cite[eq. (1.2)]{bergpeletiertroy2001}, \cite[eq. (9.0.1)]{peletiertroy}, \cite[p. 95]{peletiervivi2004}, \cite[eq. (3)]{Bonheure2004}.

\begin{enumerate}[\textbullet]
\item If $\parameter<1$, then by means of the change of variables
\[
w(s) = \frac{1}{\sqrt{1- \parameter}} U \left(\frac{1}{\sqrt[4]{1-\parameter}} s\right)
\]
we infer that $w$ solves
\begin{equation}\label{w3+w}
w''''+ \frac{2}{\sqrt{1-\parameter}}w'' + w^3 + w = 0,
\end{equation}
which is a particular case of \eqref{edo1} with $k = \frac{2}{\sqrt{1-\parameter}}$ and $f(t) = t^3+ t$. Therefore, our results apply to the equation \eqref{SHequation} with $\parameter< 1$ and, in particular, 

\begin{em}
\begin{center}
``Theorem \ref{teoremaprincipal} asserts that  in the case of $\parameter \leq0$ 
\\any nontrivial solution of \eqref{SHequation} blows up in finite space''. 
\end{center}
\end{em}

\item On the other hand, Peletier and Troy (Theorem D below) proved the existence of a nontrivial periodic solution to \eqref{SHequation} in the case of $\parameter \in (0,1)$. 

\item In the case of $\parameter=1$, no rescaling is needed as Theorem \ref{periodics} applies directly to \eqref{SHequation} and guarantees the existence of a nontrivial periodic solution.

\item If $\parameter > 1$, then by means of the change of variables
\[
\mathcal{W}(s) = \frac{1}{\sqrt{\parameter- 1}} U \left(\frac{1}{\sqrt[4]{\parameter-1}} s\right)
\]
we infer that $\mathcal{W}$ solves 
\begin{equation}\label{doublewell}
\mathcal{W}''''+ \frac{2}{\sqrt{\parameter-1}}\mathcal{W}'' + \mathcal{W}^3 - \mathcal{W} = 0.
\end{equation}
This equation, when compared with \eqref{w3+w}, presents qualitative differences as it has three equilibrium points. The nonlinear term $t^3-t$ does not satisfy our hypotheses and hence our results do not apply to this case. We refer to \cite{smetsvanderberg, Bonheure2004, santrawei2009} for results regarding existence of homoclinic and heteroclinic solutions of the equation \eqref{doublewell} and to \cite{dengli2012} for the existence of generalized homoclinic solutions of \eqref{doublewell} for every $\parameter>1$. 
\end{enumerate}

Adding all these comments together, we infer that $\parameter = 0$ is the threshold for the existence of nontrivial solution of \eqref{SHequation} that are globally defined on $\R$. More precisely: 
\begin{enumerate}[i)]
\item If $\parameter \leq 0$ then any nontrivial solution of \eqref{SHequation} blows up in finite space.
\item If $\parameter> 0$ then there exists a (bounded) solution of \eqref{SHequation} that is globally defined on $\R$.
\end{enumerate}

\end{remark}
 
To finish this introduction let us mention some earlier works that have indicated the role played  by $k_f$ in some related results. In the case of $f(t) = t^3+ t$, the number $k_f=2$ has appeared in \cite{peletiertroy}; see also \cite{colleteckmann}.
\begin{named}{Theorem D (\cite[Theorem 9.2.1]{peletiertroy})}
Consider the equation \eqref{edo1} with $f(t) = t^3+t$. If $k > k_f$, then equation \eqref{edo1} has a periodic solution with $w'(0)>0$.
\end{named}

\begin{named}{Theorem E (\cite[Theorem 9.1.1]{peletiertroy})}
Consider the equation \eqref{edo1} with $f(t) = t^3+t$. If $0 <k \leq k_f$, then there exist no nontrivial periodic solutions or homoclinic solutions of equation \eqref{edo1}.
\end{named}

Finally we recall the following result from Karageorgis and Stalker \cite[Corollary 2]{karageorgisstalker}. In particular, the theorem below extends the part of Theorem E on the nonexistence of nontrivial homoclinic solution of \eqref{edo1}.

\begin{named}{Theorem F (\cite[Corollary 2]{karageorgisstalker})}\label{kf homoclinic}
Assume that the function $f: \R \to \R$ satisfies \eqref{f1}, \eqref{hextra1} and \eqref{hextra2}. If $0 < k < k_{f}$, then \eqref{edo1} has no nonzero homoclinic solutions. 
\end{named}

\section{On the blow up of the solutions of (\ref{edo1}) and some preliminary estimates}\label{sec:naoexiste}

In this section we present a careful description of the oscillation of the solutions of the Swift-Hohenberg equation. In particular we show that, under a suitable assumptions on the parameter $k$, any nontrivial solution of \eqref{edo1} blow up. This is the first step towards the finite space blow up. We mention that  our procedure has some intersection with that adopted in \cite{gazzolapavani2013}. Neverthless, as already mentioned many of the arguments used in \cite{gazzolapavani2013} to deal with the eFK equation cannot be applied to the Swift-Hohenberg equation, so we have to overcome different difficulties. 

To start we introduce some special functions, namely $\mathcal{E}$ and $G$ below, that for example appear in \cite[Chapter 9]{peletiertroy}.

Associated to a solution $w$ of the equation \eqref{edo1} we consider the energy function
\[
\mathcal{E}(s) = w'(s)w'''(s)-\frac{1}{2} (w''(s))^2 + \frac{k}{2} (w'(s))^2 + F (w(s)) .
\]
Since $w$ solves \eqref{edo1}, we get that $\mathcal{E}'\equiv 0$. We denote by $E$ the constant value of $\mathcal{E}(s)$.

We also set
\begin{equation}\label{fcg}
G(s) = \frac{1}{2} (w''(s))^2+ \frac {k}{2} (w'(s))^2 + F(w(s)),
\end{equation}
and we denote by $H$ its derivative, namely
\[
H(s)  = G'(s) =  w''(s)w'''(s)+ k w'(s)w''(s)+ f(w(s)) w'(s).
\]
Then observe, since $w$ solves \eqref{edo1}, that
\[
H'(s) = G''(s) = (w'''(s))^2 + k w'(s)w'''(s)+ f '(w(s)) (w'(s))^2.
\]

We stress that, comparing with \cite{gazzolapavani2013}, we use the same energy function $\mathcal{E}$ but the auxiliary function $G$ and $H$ are different.

We start with a simple remark. Given a function $w$ we will denote by $\overline{w}$ the function
\[
\overline{w}(s) = w(-s).
\]
Observe that if $w$ is a solution of \eqref{edo1} in a neighborhood of $0$ then, since only even derivatives appear in the equation \eqref{edo1}, $\overline{w}$ also solves \eqref{edo1} in a possibly different neighborhood of $0$.

\begin{lemma}
Assume that $f \in C^1(\R)$ is a nondecreasing function. Then $H$ is nondecreasing and $G$ is convex, provided $|k|\leq k_f$.
\end{lemma}
\begin{proof}
Indeed, observe that $H'$ is a polynomial of degree $2$ in $w'''$. Therefore $H'$ is nonnegative provided
\begin{equation}\label{bhaskara}
D(s) = k^2 (w'(s))^2 - 4 f'(w(s)) (w'(s))^2 = (k^2-4  f '(w(s)))(w'(s))^2\leq 0,
\end{equation}
which is verified for every $s$ if $f'\geq0$ everywhere and $|k|\leq k_f$.
\end{proof}

Observe that if $|k|<k_f$, then
\[ w'(s) \neq 0 \Rightarrow D(s) < 0.\]
On the other hand, assume that \eqref{hextra2} is satisfied. If $|k|=k_f$, then 
\[w(s)w'(s) \neq 0 \Rightarrow D(s) < 0.\]

The auxiliary function $G$ was used in \cite{peletiertroy} to show that neither homoclinic nor periodic solutions exist to \eqref{edo} with $f(t) = t + t^3$ and $k \leq2$; cf. \cite[Theorem 9.1.1]{peletiertroy}. Combining the properties of $G$ and the energy function $\mathcal{E}$ we prove a preliminary result, namely Proposition \ref{teoremaexplode}, needed to the proof of Theorem \ref{teoremaprincipal}. Before that we recall the following elementary result.

\begin{remark}\label{construcaoseq}
Let $w$ be a differentiable function on $(0,\infty)$ such that
\[
a^- = \liminf_{s\to +\infty} w(s) < \limsup_{s\to+\infty} w(s) = a^+.
\] 
Then there are sequences $(x_j)$ e $(y_j)$ such that 
\[
x_j \to + \infty, \ y_j\to +\infty, \ w(x_j) \to a^-, \ w(y_j)\to a^+
\]
and
\[
w'(x_j ) = w'(y_j) = 0, \ \ w(x_j)<w(y_j), \quad \forall \ j \in \N.
\]
Here $a^-, a^+ \in [-\infty, + \infty]$.
\end{remark}

\begin{proposition}\label{teoremaexplode}
Let $0 \leq k \leq k_f$ and assume that $f$ satisfies \eqref{f1}, \eqref{f2int} and \eqref{hextra1}. Let $w$ be a nontrivial solution of (\ref{edo1}) defined in a neighborhood of $s=0$. Let $(R^-, R^+)$ be the maximal interval of existence of $w$ and set
\[a^+ = \limsup_{s \to R^+} w(s), \quad a^- = \liminf_{s \to R^+} w(s), \]
\[a_+ = \limsup_{s \to R^-} w(s), \quad a_- = \liminf_{s \to R^-} w(s). \]
Then $a^- \leq 0 \leq a^+$, $a_- \leq 0 \leq a_+$ and $w$ is unbounded. Moreover,
\begin{enumerate}[i)]
\item If $H(s_0) \geq 0$ at some $s_0\in (R^-, R^+)$, then 
\begin{equation}\label{explosao a direita}
a^+ = - a^- = + \infty.
\end{equation}
\item If $H(r_0) \leq 0$ at some $r_0 \in (R^-, R^+)$, then
\begin{equation}\label{explosao a equerda}
a_+ = - a_- = + \infty.
\end{equation}
\end{enumerate}
\end{proposition}
\begin{proof}
\noindent \textsc{Step 1: $a^- \leq 0 \leq a^+$ and $a_- \leq 0 \leq a_+$.}

\noindent If $R^+  < + \infty$, then by \cite[Lemma 23]{berchiofgk} we have $a^+ = - a^- = + \infty$. If $R^->-\infty$, then using $\overline{w}$ we get again from  \cite[Lemma 23]{berchiofgk} that $a_+ = - a_- = +\infty$. On the other hand, if $R^+ = +\infty$ then by \cite[Lemma 24]{berchiofgk} we have $a^- \leq 0 \leq a^+$. Using again \cite[Lemma 24]{berchiofgk} applied to $\overline{w}$ we have $a_-\leq 0 \leq a_+$ in the case of $R^- = -\infty$.

\medbreak
\noindent \textsc{Step 2: $w$ is unbounded.}

\noindent In the case that $R^+ < + \infty$ or $R^- > - \infty$, we know by \cite[Lemma 23]{berchiofgk} that $w$ is unbounded. So assume that $(R^-, R^+) = \R$. By contradiction suppose that $w$ is bounded, namely,
\begin{equation}\label{supondo finito}
- \infty < a^-, \ a_- \qquad \text{and } \qquad a^+, \ a_+ < + \infty.
\end{equation}

\noindent \textsc{Claim: If \eqref{supondo finito} holds true, then $G: \R \to \R$ is bounded.}

\noindent First observe that $G: \R \to \R$ being bounded leads to a contradiction. Indeed, because in this case $G$ is bounded and convex, hence constant. Therefore
$H(s) = 0$ for all $s \in \R$. Since $|k| \leq k_f$ and \eqref{hextra1} is satisfied, it follows from \eqref{bhaskara} that the product $w(s) w'(s) = 0$ for all $s \in \R$ and hence $w$ is constant. Then, from \eqref{edo1} and \eqref{f1} we get that $w \equiv 0$ on $\R$, which is a contradiction.

\medbreak
Now we prove the above claim. 
 
\noindent \textsc{Case 1: If $ a^-<a^+$, then $G$ is bounded on $[0,+ \infty)$.}

\noindent Since $a^- < a^+$, then from Remark \ref{construcaoseq} there exist two sequences $(x_n)$ e $(y_n)$, such that $x_n \to \infty$, $y_n \to \infty$ and 
\[
w(x_n) \to a^-, \ w(y_n) \to a^+, \ n\to\infty,\]
\[
w'(x_n) = w'(y_n) = 0, \ \forall \ n\in\mathbb{N}.
\]

Evaluating $G$ and $\mathcal{E}$ at a point $\xi_0$ where $w'(\xi_0) = 0$, we get 
\[
G(\xi_0) = \frac{(w''(\xi_0))^2}2 + F(w(\xi_0)), \ \ \ \mathcal{E}(\xi_0) = -\frac{(w''(\xi_0))^2}2 + F(w(\xi_0)) =E.
\]
Therefore  
\begin{equation}\label{G no ponto critico}
G(\xi_0) = 2 F(w(\xi_0)) -E, \quad \forall \ \xi_0 \ \ s.t. \ \ w'(\xi_0)=0.
\end{equation}
Then we infer that
\begin{equation}\label{finito legal}
G(x_n) \to 2 F(a^-) -E \qquad \text{and } \qquad G(y_n) \to 2 F(a^+) -E.
\end{equation}
From \eqref{supondo finito}, both limits above are finite.

If $F(a^-) \neq  F(a^+)$, then the two sequences $(x_j)$ and $(y_j)$ are such that  $x_j \to + \infty$, $y_j \to + \infty$ and 
$(G(x_j))$, $(G(y_j))$ converge to different limits, which may not occur since $G$ is convex. So we must have $F(a^-) =  F(a^+)$. Hence $G$ is bounded on $[0, + \infty)$. 
\medbreak
\noindent \textsc{Case 2: If $ a^-=a^+$, then $G$ is bounded on $[0,+ \infty)$.}

\noindent Since $a^- = a^+$, then from Step 1 it follow that $a^- = a^+ =0$. Observe that \cite[Lema 24]{berchiofgk} guarantees the existence of two sequences $(x_n)$ and $(y_n)$ such that $x_n \to + \infty$, $y_n \to + \infty$ and
\[
w'(x_n) = w'(y_n) = 0, \ \ w(x_n) < 0, \ \ w(y_n)> 0, \ \forall \ n\in\mathbb{N}.
\]
Then from \eqref{G no ponto critico} we get that
\[
G(x_n) \to -E.
\]
Then, since $G$ is convex, we conclude again that $G$ is bounded on $[0, +\infty)$.

\medbreak

From Cases 1 and 2, we conclude that $G$ is bounded on $[0, +\infty)$. Now applying the same argument to $\overline{w}$ we conclude that $G$ is also bounded on $(- \infty, 0]$. Therefore, the above claim is proved. Hence, at this point we have proved that $w$ is not bounded, that is
\[
a^++a_+-a^--a_- = +\infty.
\]

\medbreak
\noindent \textsc{Step 3: Proof of $i)$.}

\noindent If $R^+ < + \infty$, then \eqref{explosao a direita} follows by \cite[Lemma 23]{berchiofgk}. Now, suppose that $R^+ = + \infty$ and $H(s_0) \geq 0$ for some $s_ 0\in (R^-, R^+)$. Then there exists $s_1 \in [s_0, + \infty)$ such that $H(s_1)>0$. Otherwise, as we argued above, we would get $H(s)=0$ for every $s\geq s_0$, which implies $w(s) w'(s) = 0$ for every $s \geq s_0$, then $w \equiv 0$ on $[s_0, + \infty)$ and hence on $(R^-, R^+)$, which is a contradiction. Then, since $G$ is convex we get 
\[
G(s) \geq G(s_1) + (s-s_1) H(s_1), \quad s \geq s_1.
\]
Then $G(s)\to +\infty$ as $s\to +\infty$. Then, as we argued at Cases 1 and 2 above, $w$ cannot be bounded on $[0, +\infty)$ neither from below nor from above (otherwise $G$ should be bounded on $[0, + \infty)$). 

\medbreak
\noindent \textsc{Step 4: Proof of $ii)$.}

\noindent If $R^- > - \infty$, then \eqref{explosao a direita} follows by \cite[Lemma 23]{berchiofgk}.In the case that $H(r_0)\leq0$ for some $r_0 \in (R^-, R^+ )$ then consider $G$ the function associated to $\overline{w}$. In this case the function $H$ associated to $\overline{w}$ will be nonnegative at some point and hence we apply the argument form the last paragraph.
\end{proof}

\begin{remark}
Let $0 \leq k \leq k_f$ and assume that $f$ satisfies \eqref{f1}, \eqref{f2int} and \eqref{hextra1}. Let $w$ be a nontrivial solution of \eqref{edo1} and assume that $H(s_0)\geq0$ at some $s_0$.  Then $H(s)> 0$ for every $s>s_0$. In particular, if $H(0)\geq 0$ then there exists $m>0$ such that $m$ is a local maximum of $w$, $w(m)>0$ and $H(m)>0$.
\end{remark}

Next we give detailed information about the oscillations of a solution $w$ of \eqref{edo1}. We stress that these information are crucial in proof of the finite space blow up of solutions of \eqref{edo1}. We emphasize that we are dealing with the case of $k >0$ and that we are forced to argue differently from \cite[Lemmas 10 and 11]{gazzolapavani2013}.

\begin{lemma} \label{lema10}
Let $0 \leq k \leq k_f$ and assume that $f$ satisfies \eqref{f1}, \eqref{f2int} and \eqref{hextra1}. Let $w = w(s)$ be a local solution of \eqref{edo1} on a neighborhood of $s =0$ and defined on the maximal interval on the right $[0,R)$. Let $m\geq 0$ be a local maximum of $w$ with $w(m)>0$ and $H(m)>0$. Let $m'>m$ be the next critical point of $w$. Then:
\begin{enumerate}[i)]
\item There exists $r \in (m, m')$ such that $w''<0$ on $[m, r)$ and $w''> 0$ on $(r,m']$. Furthermore, $w(r)<0$ and $w'<0$ on $(m,m')$. In particular $w(m')< 0$ and $m'$ is local minimum of $w$.
\item There exists $z \in (m, m')$ such that $w(z)=0$. 
\item There exists $\tau \in (m,m')$ such that $w'''< 0$ on $[m,\tau)$ and $w'''>0$ $(\tau,m']$. In addition, $z<\tau<r$.
\end{enumerate}
\end{lemma}
\begin{proof} 
\noindent \textit{Proof of i) and ii) } Since $w$ is a nontrivial solution of \eqref{edo1} we know that the critical points of $w$ are isolated. Observe that \eqref{explosao a direita} guarantees that $w$ has infinitely many critical points greater than $m$. So $m'$, the next critical point of $w$ greater than $m$, is well defined. Moreover, since $m$ is a local maximum of $w$, we have $w'<0$ on $(m,m')$ and $w'(m')=0$.
By hypothesis $H(m) = w''(m)w'''(m) > 0$. Since $m$ is a local maximum, we infer that $w''(m)<0$ and $w'''(m)<0$. Then there exists $\varepsilon > 0$ such that $w''<0$ on $[m, m+\varepsilon)$.

Set $r = \sup \{s>m; w''<0$ on $[m, s) \}$. Observe that Proposition \ref{teoremaexplode}, namely \eqref{explosao a direita}, guarantees that $w$ has a local maximum at some point $\xi> m$. Hence, $r< + \infty$ and $w''(s)< 0$ on $[m,r)$ and $w''(r)= 0$.

Since $m'> m$ is the next critical point of $w$ after $m$, it follows that $m< r< m'$.
 
Next we prove that $w(r)<0$. Indeed, since $H(r) = f(w(r))w'(r) > 0$ and $w'(r) < 0$ we infer that $f(w(r))<0$ and hence that  $w(r) <0$. Consequently we also have $w(m')< 0$ and there exists $m < z< r$ such that $w(s)>0$ on $[m,z)$ and $w(s)<0$ on $(z,m']$.

Now we prove that $w''> 0$ on $(r,m']$. Observe that $w'<0$ on $(m, m')$, $w'(m') =0$ and $H(m')= w''(m') w'''(m')>0$. Hence $w''(m')>0$ and $w'''(m')>0$. Therefore, $m'$ is a local minimum of $w$. 

Observe that if there exists $\xi \in (m, m')$ such that $w'''(\xi)=0$, then
\[ 
0 < H(\xi) = -w'(\xi)w{''''}(\xi), 
\]
and then $w''''(\xi)> 0$. This implies that all the critical points of $w''$ on $(m,m')$ are strict local minima. Therefore 
\begin{equation}\label{unicidade w''' igual a zero}
w''' \ \ \text{vanish at most at one point in}  \ \ (m,m').
\end{equation}

We recall that, by definition, $w''(r)=0$ and $w''<0$ em $[m,r)$. Then $w'''(r) \ge 0$. Since $w'''(m)< 0$ and $w'''(m')> 0$, we know that there exists $\tau \in (m,r]$ such that $w'''(\tau) =0$. Moreover, from \eqref{edo1},
\[ w{''''}(r) = - f(w(r))>0, \]
which implies $w'''>0$ and $w''>0$ on $(r, r+\varepsilon)$ and  for some $\varepsilon >0$. Set 
\[
r_* = \sup \{s > r; w''>0 \ \text{em} \ (r,s) \}.
\]

Then by \eqref{explosao a direita} we know that $r< r_*< + \infty$, $w''(r_*)=0$
and $w''>0$ em $(r, r_*)$. Hence, $w'''(r_*) \le 0$. By contradiction suppose that $r_*<m'$.  Then $w'(r_*)<0$ and from
\[
0 < H(r_*) = f(w(r_*)w'(r_*) = - w''''(r_*)w'(r_*) 
\]
we infer that $w{''''}(r_*)>0$.  Then $w'''<0$ on $(r_*- \varepsilon_*, r_*)$ for some $\varepsilon_*>0$. Therefore, $w''(r) = w''(r_*) = 0$ and there exist $\tau_*\in (r, r_*)$ such that
\[
0 < \max_{s \in [r, r_*]} w''(s) = w''(\tau_*).
\]
Hence $w'''(\tau_*) = 0$. Therefore, there exist $\tau \in (m, r]$ and $\tau_* \in (r, m')$ such that $w'''(\tau) = w'''(\tau_*) =0$, which contradicts \eqref{unicidade w''' igual a zero}.

\medbreak
\noindent \textit{Proof of  iii) } We have shown that there exits $\tau \in (m, r]$ such that $w'''(\tau) = 0$. We know that $w'<0$ em $(m, m')$. Then from \eqref{unicidade w''' igual a zero} we conclude that
\[
w''' < 0 \ \ \text{on} \ \ [m, \tau) \ \ \text{and} \ \ w'''> 0 \ \ \text{on} \ \ ( \tau, m'].
\]

Now we prove that $z < \tau < r$. Indeed we know $w''(z) < 0$, $w''(r) =0 $ and $w'''(z) < 0$. Hence the minimum of $w''$ on $[z,r]$ is attained on a point in $(z,r)$ which is necessarily $\tau$.
\end{proof}

\begin{remark}\label{obs sinais contrarios}
Similar properties holds on any interval on the right of a local minimum $m$ of $w$ such that $w(m)< 0$ and $H(m)> 0$. In this case the function $w,w', w''$ and $w'''$ will have the inverse sign. From Lemma \ref{lema10}, we obtain the Fig. \ref{graficow} below and have described the behavior of a solution $w$ of \eqref{edo1} in between two consecutive local maxima.
\end{remark}
\begin{figure}[h]\label{graficow}
\begin{center}
\includegraphics[width=0.8\textwidth]{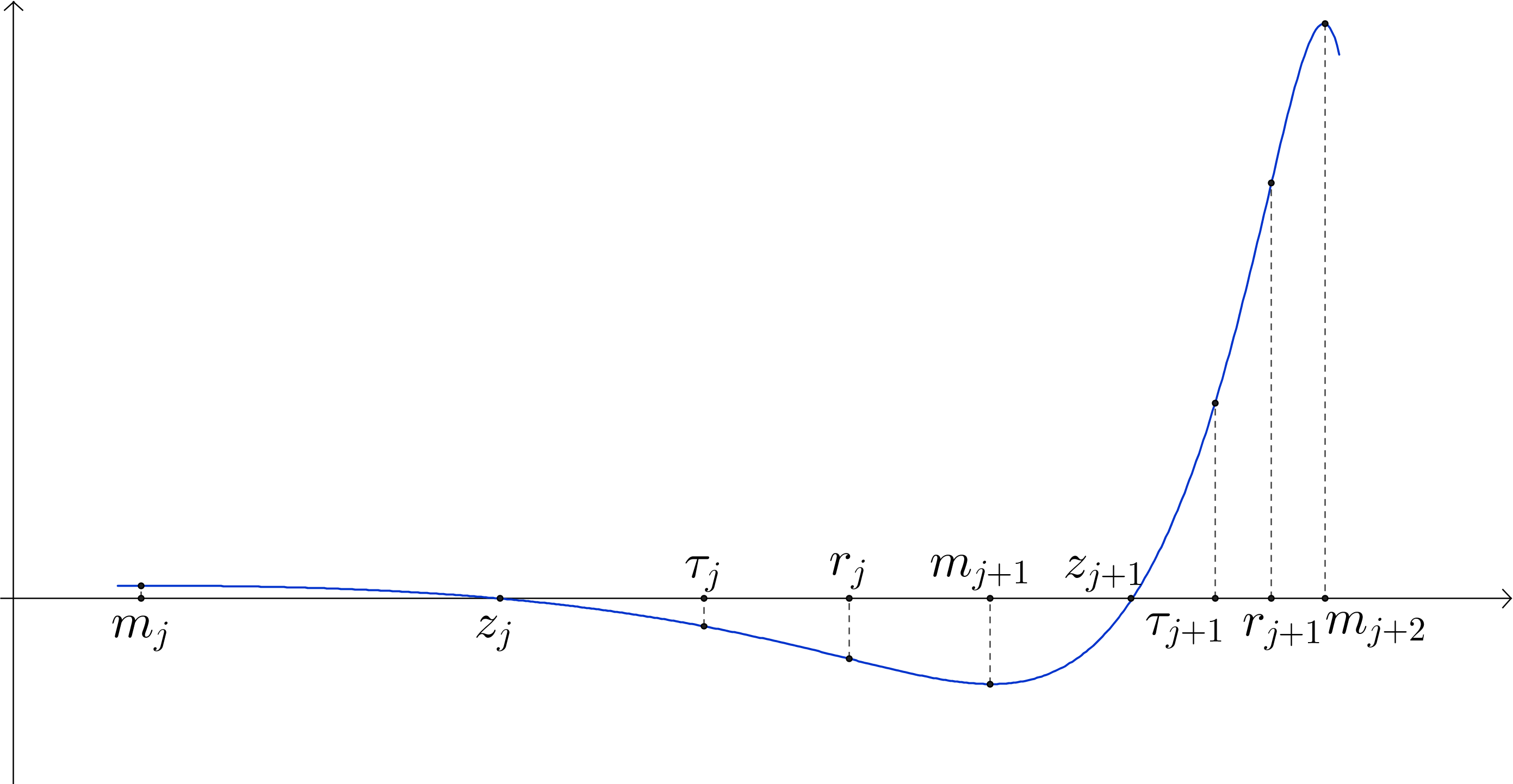}
\end{center}
\vspace{-.6cm}
\caption{Behavior of $w$ in between two local maxima}
\end{figure}

To our purposes we also need to have information about the fourth and fifth derivatives of $w$.
\begin{lemma}\label{lema quarta e quinta derivadas}
Let $0 < k \leq k_f$ and assume that $f$ satisfies \eqref{f1}, \eqref{f2int}, \eqref{hextra1} and \eqref{hextra2}. Let $w = w(s)$ be a local solution of \eqref{edo1} on a neighborhood of $s =0$ and defined on the maximal interval on the right $[0,R)$. Let $m\geq 0$ be a local maximum of $w$ with $w(m)>0$ and $H(m)>0$. Let $z, \tau, r, m'$ as in Lemma \ref{lema10}. Then {\rm{(}}since $w(m)$ becomes larger and larger by \eqref{explosao a direita}{\rm{)}} we have:
\begin{enumerate}[i)]
\item $w^{(v)} > 0$ on $[m,\tau]$.
\item There exists $m<\theta<z$ such that $w''''< 0$ on $[m, \theta)$ and $w''''> 0$ on $(\theta,r]$.
\end{enumerate}
\end{lemma}
\begin{proof}
If we differentiate \eqref{edo1} we get 
\begin{equation}\label{edo'}
w^{(v)} (s)= -kw'''(s)-f'(w(s))w'(s).
\end{equation}
From Lemma \ref{lema10} we know that the two terms on the right hand side of \eqref{edo'} are positive in $(m, \tau)$ and we infer that $w^{(v)}(s) > 0$ on $[m,\tau]$. Therefore $w''''$ has at most one zero on $[m , \tau]$. Since $w{''''}(z) = -kw''(z)  >0$ we infer that $w{''''}>0$ on $[z,\tau]$.

Observe that, from \eqref{explosao a direita}, we already know that $a^+ = + \infty$. So, with no loss of generality  we may assume that $w(m)>>1$. Then from  the energy function $\mathcal{E}$ and from \eqref{f2int} we get
\begin{equation} \label{d1}
\frac{\rho}{p+1}(w(m))^{p+1} - E \! \leq \! \frac{(w''(m))^2}{2} \! = \! F(w(m)) - E \le  \! \!\left(\frac{\alpha}{q+1} + \frac{\beta}{p+1}\right) \!(w(m))^{p+1} - E,
\end{equation}
and hence there exist constants $C_1 > 0$ and $C_2> 0$ such that
\begin{equation}\label{d2}
C_1 (w(m))^{\frac{p+1}{2}} \leq |w''(m)| \leq C_2 (w(m))^{\frac{p+1}{2}}.
\end{equation}

Now we estimate $w{''''}(m)$. From (\ref{edo1}) and \eqref{f2int} we get
\begin{equation}\label{sinal quarta derivada maximo}
w{''''}(m) = -f(w(m)) - kw''(m)  \le - \rho (w(m))^p + k\,C_2 (w(m))^{\frac{p+1}{2}}.
\end{equation}
Therefore, since $w(m)> >1$ and $p>1$, we have that $w{''''}(m)<0$. Then we conclude that there exists a unique $\theta \in (m,z)$ such that $w{''''}<0$ on $[m, \theta)$ and $w{''''}>0$ on $(\theta, \tau]$. 

Now, from \eqref{edo1} we obtain that $w''''$ may not vanish on $[\tau,r]$ because $w''(s)$ and $f(w(s))$ are both negative on this interval. Therefore, we infer that $w''''>0$ on $[m,r]$. 
\end{proof}

\begin{remark} Assume all the hypotheses from Lemma \ref{lema quarta e quinta derivadas} and let $m< \theta < z < \tau<  r<  m'$ be as in Lemmas \ref{lema quarta e quinta derivadas} and \ref{lema10}. From Lemma \eqref{lema quarta e quinta derivadas} ii) we know that $w''''>0$ on $[\tau, r]$. Therefore we get that
\[
w'''(r) > - \frac{w''(\tau)}{r - \tau} > 0
\]
and hence
\begin{equation}\label{rel terceira e segunda em r e tau}
 (w'''(r))^2 > \frac{(w''(\tau))^2}{(r- \tau)^2} = \frac{G(\tau) - E}{(r- \tau)^2}.
\end{equation}

On the other hand, from \eqref{edo1},
\[
w{''''}(s)+ k w''(s) = - f(w(s)) > 0 \quad \text{on} \quad [z, m'] 
\]
and so
\[
w'''(s) w''''(s) + k w''(s)w'''(s) = - f(w(s)) w'''(s) > 0 \quad \text{on} \quad (\tau, m').
\]
Hence the function
\[
s \mapsto (w'''(s))^2 + k (w''(s))^2 \quad \text{is strictly increasing on} \quad [\tau, m']  
\]
and, in particular,
\begin{equation}\label{rel em r e m'}
(w'''(m'))^2 + k (w''(m'))^2 > (w'''(r))^2.
\end{equation}
Therefore, from \eqref{rel terceira e segunda em r e tau} and \eqref{rel em r e m'} we get
\begin{equation}\label{des importante1 para blow up}
(w'''(m'))^2 + k (w''(m'))^2 > \frac{G(\tau) - E}{(r - \tau)^2}.
\end{equation}
Hence, since $G(s) \to + \infty$ as $s \rt R^+$, we have $G(\tau)- E>0$.
\end{remark}

\section{On the finite space blow up: Proof of Theorem \ref{teoremaprincipal}}

We recall that if $w$ is a solution of \eqref{edo1} then so is $\overline{w}(s) = w(-s)$. Moreover $H_{w}(0) = - H_{\overline{w}}(0)$. Therefore, in Theorem \ref{teoremaprincipal}, the case $ii)$ follows from the case i).

Let $w = w(s)$ be a local solution of \eqref{edo1} on a neighborhood of $s =0$ and defined on the maximal interval on the right $[0,R)$. Assume also that $H(0)\geq 0$. Since $0 < k \leq k_f$, we know that $H'(s)\geq 0$ for every $s$. Then, since $w$ is nontrivial a nontrivial solution of \eqref{edo1}, we infer that $H>0$ on $(0,R)$.

In this part we have to prove some technical estimates. The proof of some of these estimates are similar to some in \cite[Section 7]{gazzolapavani2013} and in this case we will refer to \cite{gazzolapavani2013} for more details. We will be careful to mention the similarities between the computations below and those in \cite{gazzolapavani2013} as well as to stress the distinct and crucial ones. 

\medbreak\noindent\textsc{Step 1: Construction of the sequences $(m_j)$, $(z_j)$, $(\tau_j)$ and $(r_j)$.}

\noindent From Proposition \ref{teoremaexplode}, Lemma \ref{lema10} and Remark \ref{obs sinais contrarios} we obtain the sequence $(m_j)$ of all the critical points of $w$ after a certain fixed $s_0 \geq0$ in such way that
\begin{align*}
w(m_{2k-1}) < 0, \quad w(m_{2k}) > 0, \quad w''(m_{2k-1}) > 0, \quad w''(m_{2k}) < 0,\\
m_{j} \to R, \quad w(m_{2k-1}) \to - \infty \quad \text{and}  \quad w(m_{2k}) \to + \infty.
\end{align*}
Moreover there are sequences $(z_j)$, $(\tau_j)$ and $(r_j)$ such that
\[ m_j < z_j< \tau_j < r_j < m_{j+1} \quad \forall \ j \in \mathbb{N} \]
in such way that $(z_j)$, $(m_j)$, $(r_j)$ and $(\tau_j)$ are, respectively, the sequences of all zeros of $w$, $w'$, $w''$ and $w'''$. We also recall that
\[
w''(m_j) w'''(m_j) = H(m_j) >0 \quad \forall \ j \in \N.   
\]

We set $M_j = |w(m_j)|$ for $j \in \N$. 

\medbreak\noindent\textsc{Step 2: There exists $C>0$ such that}
\begin{equation}\label{m--z} z_j - m_j \le C M_j^{\frac{1-p}{4}}. \end{equation}

This estimate is a consequence of the concavity/convexity of $w$ on $[m_j, z_j]$, which was proved in Lemma \ref{lema10}. Then the arguments in \cite[Step 4 at p. 743]{gazzolapavani2013} apply. We stress that although $k$ is negative in \cite{gazzolapavani2013}, the estimates in this part only involves $|k|$.

\medbreak\noindent\textsc{Step 3: The limit}
\begin{equation}\label{limitezeroo} \lim_{j \to \infty} r_j - z_j = 0 \quad \textsc{holds}.\end{equation}
Here our arguments are slightly different from those in \cite[Step 2 at p. 740]{gazzolapavani2013}. The main difference is that our auxiliary function $G$ is not the same as in \cite{gazzolapavani2013}. 

We consider the case that $w$ has a local maximum at $m_j$, which according to our notation corresponds to any even $j$. The same arguments apply to the case that $w$ has local minimum at $m_j$.

Suppose by contradiction that 
\[
\limsup\limits_{j \to \infty}(r_j-z_j)>0. 
\]
Then there exists $a>0$ and a subsequence (denoted with the same index) such that $r_j - z_j \ge a$, for all $ j>j_0$. Set the function $h(s)=(s-z_j)^3(z_j+a-s)^4$. Multiply \eqref{edo1} by $h$. Then simple integration by parts yields 
\begin{multline}\label{lim1}
0 = \int_{z_j}^{z_j+a} h(s) (w{''''}(s) + kw''(s)+ f(w(s)) ds \\= \int_{z_j}^{z_j+a} h(s)f(w(s)) + (k h''(s)+h{''''}(s))w(s) ds.
\end{multline}
On the other hand, since $w$ is strictly concave on $[z_j, r_j]$, $ w'$ is decreasing, and
\[ w(z_j+s) = \int_0^s w'(z_j+t) dt \ge w'(z_j)s, \quad \forall \ s \in (z_j, r_j).\]
Once more using that $w$ is concave on $(m_j, z_j)$, we infer that
\[ w'(z_j) \ge \frac{M_j}{z_j-m_j}. \] 
Then, from Step 2, we get
\[ w'(z_j) \ge CM_j^{1+ \frac{p-1}{4}} = CM_j^{\frac{p+3}{4}},\]
and that
\begin{equation}\label{lim2}
w(z_j + s) \geq C M_{j}^{\frac{p+3}{4}}s, \quad \forall \ s \in (z_j, r_j).
\end{equation}
On the other hand, observe that
\[ h(s)f(w(s))+(kh''(s)+h{''''}(s)) w \ge h(s)f(w(s)) - Cw(s), \]
with $C=  \sup\limits_{[z_j,a]} |kh''(s)+h{''''}(s)|$ that does not depend on $j$. Then, for every $0< \varepsilon < a$, we have
\begin{multline}\label{lim3}
\int_{z_j + \varepsilon}^{z_j + a} h(s)f(w(s)) ds\ge h(z_j + \varepsilon ) \int_{z_j + \varepsilon}^{z_j + a} w (z_j + s)^p ds\\\ge (a - \varepsilon) h (z_j + \varepsilon ) C \varepsilon^p M_j^{\frac{p(p+3)}{4}}.
\end{multline} 
Then from \eqref{lim3},
\begin{multline*}
\int_{z_j }^{z_j + a} h(s)f(w(s)) + (kh''(s)+h{''''}(s))w(s) ds  \ge  \int_{z_j}^{z_j+a} h(s)f(w(s))-Cw(s) ds \\
 \ge  \int_{z_j + \varepsilon }^{z_j + a} h(s)f(w(s)) - C \int_{z_j }^{z_j + a} w(s)ds 
 \ge  (a - \varepsilon ) \varepsilon ^p Ch(z_j + \varepsilon )  M_j^{\frac{p(p+3)}{4}} - C \, a \,  M_j. 
\end{multline*}

Then, for $j$ suitably large, the integral becomes positive, which yields a contradiction.

\medbreak\noindent\textsc{Step 4: There exists $C>0$ such that}
\begin{equation}\label{z--r} r_j - z_j \le C M_j^{\frac{1-p}{4}}. \end{equation}
Since \eqref{limitezeroo} is established, the argument follows as in \cite[Step 3 at p. 741]{gazzolapavani2013}.

\medbreak\noindent\textsc{Step 5: There exists $C>0$ such that}
\begin{equation}\label{r--m} m_{j+1} - r_j \le C M_j^{\frac{1-p}{4}}. \end{equation}
The proof of this step follows as in \cite[Step 5 at p. 745]{gazzolapavani2013}. We stress that the information about $w''$ given at Lemma \ref{lema10} is crucial here.

\renewcommand{\thefootnote}{\fnsymbol{footnote}}

\medbreak\noindent\textsc{Step 6: We show that $R<\infty$.}

\noindent We mention that plan to prove this step is the same as in \cite{gazzolapavani2013}. However, the main part of the arguments in this step is very different from those in \cite[Step 6 p. 747]{gazzolapavani2013}. The reason is that in the case of $k>0$ we cannot guarantee that the function
\[
\Phi(s) = \frac{(w''(s))^2}{2} + F(w(s))
\]
is convex\footnote{In the case of $k\leq 0$ it is easy to verify that $\Phi$ is a convex function.}. 

Since $G$ is increasing, we infer that $(G(m_j))$ is an increasing sequence. Then
\begin{equation}\label{l3}
F(w(m_j)) = \frac{G(m_j)- E}{2} < \frac{G(m_{j+1})-E}{2} = F(w(m_{j+1})), 
\end{equation}
which shows that $(F(w(m_j)) )$ is also increasing.

To finish the proof, as we will see below,  it is enough to show that there exists $i_0 \in \N$ such that
\begin{equation}\label{l1}
F(w(m_{i+1})) > 2 F(w(m_{i-1})), \quad \forall \ i \geq i_0.
\end{equation}

We stress that we were not able to prove that
\[
F(w(m_{i})) > 2 F(w(m_{i-1})), \quad \forall \ i \geq i_0,
\]
for some sufficiently large $i_0$. The reason is that we could not prove that $|w'''(m_i)| > |w''(m_i)|$ for all $i$ sufficiently large. Nevertheless, we will see that \eqref{l1} is sufficient to prove that $R< +\infty$.

\medbreak\noindent\textsc{Case 1: Assume that $(w'''(m_{i}))^2 \leq (w''(m_{i}))^2$. Then \eqref{l1} holds.}

\noindent From \eqref{des importante1 para blow up}, we infer that
\[
(1 + k) (w''(m_{i}))^2 \geq \frac{G(\tau_{i-1}) - E}{(r_{i-1} - \tau_{i-1})^2} > \frac{G(m_{i-1})- E}{(r_{i-1} - \tau_{i-1})^2} = \frac{(w''(m_{i-1}))^2}{(r_{i-1} - \tau_{i-1})^2}
\] 
and hence
\[
(1+k) 2 (F(w(m_{i})) - E) > \frac{2 (F(w(m_{i-1})) - E)}{(r_{i-1} - \tau_{i-1})^2},
\]
which yields
\[
F(w(m_{i})) > 2 F(w(m_{i-1}))
\]
for every $i$ sufficiently large, because $F(w(m_j)) \to + \infty$ and $|r_{i-1} - \tau_{i-1}| \to 0$ as $j \rt + \infty$. Then,  since $F(w(m_j))$ is increasing, we get that
\[
F(w(m_{i+1})) > 2 F(w(m_{i-1})) 
\]
which is precisely \eqref{l1}.
\medbreak\noindent\textsc{Case 2: Assume that $(w'''(m_{i}))^2 > (w''(m_{i}))^2$. Then \eqref{l1} holds.}

\noindent We define
\[ N_0 = \{ j \in \mathbb{N}; 2F(w(m_{j})) \le F(w(m_{j+1})) \ \ \text{and} \ \ (w'''(m_{j}))^2 > (w''(m_{j}))^2 \}, \]
\[ N_1 = \{ j \in \mathbb{N}; 2F(w(m_{j})) > F(w(m_{j+1})) \ \ \text{and} \ \ (w'''(m_{j}))^2 > (w''(m_{j}))^2\}. \]

Observe that if $i \in N_0$ then we are done because $F(w(m_j))$ is increasing. Hence we will prove that $N_1$ is bounded.

First observe that if $\Phi (m_{j}) < \Phi (z_{j})$ is satisfied then $j \in N_0$\footnote{In the case of $k\leq 0$, the inequality $\Phi (m_j) < \Phi (z_j)$ is always satisfied. As a consequence $N_1 = \emptyset$.}. Indeed,
\begin{align}\label{l4}
2F(w(m_{j})) & = \Phi (m_{j}) + E < \Phi (z_{j}) + E = \frac{(w''(z_{j}))^2}{2} + E \nonumber \\
& < \frac{(w''(\tau_{j}))^2}{2} + E = \frac{G(\tau_{j})}{2} + \frac{E}{2} < \frac{G(m_{j+1})}{2} + \frac{E}{2}= F(w(m_{j+1})).
\end{align}

Now let $j \in N_1$. Then
\begin{equation}\label{contra}
F(w(m_{j+1})) < 2F(w(m_{j})).
\end{equation}
Since $G$ is convex, we infer that
\begin{equation}\label{final1}
G(z_{j}) > G(m_{j})+ H(m_{j})(z_{j}-m_{j}),
\end{equation}
which we rewrite as
\begin{align}\label{final2}
\Phi (z_{j}) + \frac{k(w'(z_{j}))^2}{2} & > \Phi (m_{j}) + w''(m_{j}) w'''(m_{j}) (z_{j}-m_{j}) \nonumber \\
& \ge \Phi (m_{j}) + w''^2 (m_{j})(z_{j}-m_{j}),
\end{align}
where the last inequality comes from the hypothesis $(w'''(m_{j}))^2 > (w''(m_{j}))^2$.

\medbreak\noindent\textsc{Claim: For every sufficiently large $j$
\noindent
\begin{equation}\label{final3}
\frac{k(w(z_j))^2}{2} < (w''(m_j))^2(z_j-m_j).
\end{equation}
}

We consider the case that $w$ has a local maximum at $m_j$, which according to our notation corresponds to any even $j$. The same arguments apply to the case that $w$ has local minimum at $m_j$.

Indeed, from the concavity of $w'$, we have
\[ w''(z_j) < \frac{w'(z_j)}{z_j-m_j}, \]
which yields
\begin{equation}\label{final4}
\frac{k}{2} (w''(z_j))^2(z_j-m_j)^2 > \frac{k}{2} (w'(z_j))^2.
\end{equation}
On the other hand,
\begin{multline}\label{hhhhh}
\frac{k}{2} (z_j-m_j)(w''(z_j))^2  < \frac{k}{2} (z_j-m_j)(w''(\tau_j))^2 = \frac{k}{2} (z_j-m_j)(G(\tau_j)-E) \\
 < \frac{k}{2}(z_j-m_j) (G(m_{j+1})-E) = k (z_j - m_j)(F(w(m_{j+1}))-E) \\
 < k(z_j-m_j)(2F(w(m_j))-E) = k (z_j - m_j) \frac{(w''(m_j))^2+E}{(w''(m_j))^2} (w''(m_j))^2.
\end{multline}
Then observe that
\[ k(z_j-m_j) \frac{(w''(m_j))^2+E}{(w''(m_j))^2} \to 0 \quad \text{as} \quad j \rt +\infty.\]
Then, from \eqref{hhhhh} and \eqref{final4}, there exists $j_0 \in \mathbb{N}$ such that 
\[ \frac{k(w'(z_j))^2}{2} < \frac{k}{2} (w''(z_j))^2(z_j-m_j)^2 < (w''(m_j))^2(z_j-m_j) \quad \forall \ j \geq j_0, \]
which proves \eqref{final3}.

Hence, if $j \geq j_0$ then \eqref{final2} and \eqref{final3} yield $\Phi(z_j) > \Phi (m_j)$, which as observed above implies that $j \in N_0$, which then leads to a contradiction. Therefore $N_1$ is bounded.

\medbreak At this point we have concluded the proof of \eqref{l1}.

\medbreak
So, using the notation $M_j = |w(m_j)|$, we get from \eqref{l1} that 
\[
F(M_{2j+j_0})>2^jF(M_{j_0}) \quad \forall \ j \geq1.
\]

On the other hand, by \eqref{f2int},
\[ F(M_j) \le \frac{\alpha}{q+1} M_j^{q+1} + \frac{\beta}{p+1} M_j^{p+1} = \frac{\beta}{p+1} M_j^{p+1}(\alpha\beta^{-1}\frac{p+1}{q+1} M_j^{q-p}+1). \]
Then, since $q-p<0$, there exists a positive constant $C$ such that
\begin{equation}\label{FmenorM}
(F(M_j))^{-1} \ge C M_j^{-(p+1)} \quad \forall \ j \geq j_0.
\end{equation}

Now we use the inequalities from Step 2, 4 and 5 and the inequality \eqref{FmenorM}. In addition, we recall that $(F(w(m_j)))$ is an increasing sequence. Then we infer that
\begin{multline*}
m_{2j+j_0+1} - m_{j_0+1}  =  \sum_{l=1}^j m_{j_0+2l+1}-m_{j_0+2l-1} \le C \sum_{l=1}^j ( M_{j_0+2l}^{\frac{1-p}{4}}+ M_{j_0+2l-1}^{\frac{1-p}{4}} ) \\ \le C \sum_{l=1}^j ( F(M_{j_0+2l})^{\frac{1}{4}\frac{1-p}{p+1}} + F(M_{j_0+2l-1})^{\frac{1}{4}\frac{1-p}{p+1}} )  \leq 2C \sum_{l=1}^j  F(M_{j_0+2l-2})^{\frac{1}{4}\frac{1-p}{p+1}} \\
 \le 2 C \sum_{l=1}^j (F(M_{j_0})2^{l-1})^{\frac{1}{4}\frac{1-p}{p+1}} = \bar{C} \sum_{l=1}^j (2^{\frac{1}{4}\frac{1-p}{p+1}})^{l-1}. 
\end{multline*}

Then taking the limit as $j \to \infty$ we get that $m_{2l+j_0+1}-m_{j_0+1} \to R-m_{j_0+1}$ and the series from the right hand side converges because $p>1$. Therefore we conclude that $R < + \infty$.

\section{Periodic solutions beyond the threshold $k_f$: The proof of Theorem \ref{periodics}}

\renewcommand{\a}{a}
\renewcommand{\b}{b}
\newcommand{\q}{\sigma}
\renewcommand{\r}{r}
In this part we prove the existence of periodic solutions to \eqref{edo1} by using the topological shooting technique. Here we use some ideas developed in \cite{peletiertroy1998}. We also mention that similar results were proved in \cite[Chapter 9]{peletiertroy} for the particular case of $f(t) = t + t^3$ and we stress that some different arguments are needed, in particular, to include the case of $f'(0)=0$. 

Consider the initial value problem 
\begin{equation}\label{pvisegundo}
\left\{
\begin{array}{l}
w{''''}+k w'' + f(w) =0,\\
(w,w',w'',w''')(0)=(0, \a, 0, \b).
\end{array}\right.
\end{equation}
For some $\a, \b\in\mathbb{R}$ given, if there exist two critical points $m>0>m^*$ of the solution $w$, such that 
\begin{equation}\label{mms}
w'(m)=w'(m^*) = 0, \quad w'''(m)=w'''(m^*)=0,
\end{equation}
then the extension of $w$ to $[2m^*,2m]$ defined by 
\[
w(s)=\left\{\begin{array}{rl}
w(2m-s), & m\leq s <2m\\
w(s), & m^*\leq s<m\\ 
w(2m^*-s), & 2m^*\leq s \leq m^*
\end{array}\right.
\]
solves \eqref{pvisegundo} on $[2m^*,2m]$. From this definition, $w$ and its derivatives up to the third order coincide at $2m$ and $2m^*$. Then by the unicity of solution of \eqref{pvisegundo}, $w$ is a periodic solution of \eqref{pvisegundo} as required with period $2m-2m^*$. 

If $f$ is odd, and there exists $m>0$ such $w'(m) = w'''(m) = 0$, then the odd extension to $[-m,m]$ solves \eqref{pvisegundo}, and we may take $m^*=-m$ to get \eqref{mms}, thus obtaining a periodic solution of \eqref{pvisegundo}. Such a point $m$ is called a \textit{point of symmetry} of $w$.

Notice that $\a$ and $\b$ are related by means of the energy function, namely,
\[  E = \a \b + k \a^2/2. \]
If $\a = 0$, then necessarily $E=0$ and $\b$ is a free parameter. On the other hand, if $\a\neq 0$, we can write
\[ \b = \b(\a ) = \frac E\a- \frac{k\a} 2 \]
and take $\a$ as parameter (given an energy value $E$). 
In what follows we will consider the latter case.

\begin{remark}${}$

\begin{enumerate}[a)]
\item The solution that we construct in the proof of Theorem \ref{periodics} has exactly two critical points in each period. Such a solution is called a single-bump solution.
\item From the choice of initial value, the solution $w$ of \eqref{edo1} satisfies $H(0)=0$ and $H'(0)>0$. The function $H$ no longer needs to be monotone as $k>k_f$.
\end{enumerate}
\end{remark}
\medbreak
\begin{proof}[\textbf{Proof of Theorem \ref{periodics}}]
Given $\a>0$ take $b = - k a/2$. Denote by $w(s,\a )$ the value of the solution of \eqref{pvisegundo} at the point $s\geq 0$. Consider 
\[  m(\a ) = \sup \{t>0 ; \ w'(s,\a )>0, \ \forall\ s\in[0,t)\}.  \]
By \cite[Theorem 4]{berchiofgk} or \cite[Lemma 24]{berchiofgk}, according to whether or not the solution is globally defined, we know that $w$ changes sign infinitely many times, and so $0< m(\a ) < + \infty$ and  $w'(m(\a ), \a )=0$.

We just need to verify that the function
\[ \phi(\a ) = w'''(m(\a ), \a )
\]
has a root on $(0,+\infty)$. Observe that $\phi$ is continuous inasmuch as the solution depends continuously on the initial data. From Lemma \ref{agrande} $\phi(\a )<0$ for large values of $\a$, whereas by Lemma \ref{apequeno}, $\phi(\a )>0$ if $\a$ is sufficiently small and these concludes the proof of Theorem \ref{periodics}.
\end{proof}

\begin{lemma}\label{agrande} If $f$ satisfies \eqref{f1} and \eqref{f2int}, then there exists $\a_1>0$, such that $\a>\a_1$ implies  ${w'''(m(\a ),\a )<0}$.
\end{lemma}

\begin{proof}
Given $\q, \r\in\mathbb{R},$ consider the change of variables 
\[ t = \a^\q s, \quad v(t) = \a^\r w(s). \]
From \eqref{pvisegundo} we get 
\begin{equation}\label{pvialpha}
\left\{
\begin{array}{l}
v{''''}+k \a^{-2\q}v'' + \a^{\r-4\q}f(\a^{-\r}v) =0, \\ 
(v,v',v'',v''')(0)=(0, \a^{\r-\q+1}, 0, \a^{\r-3\q}\b(\a)).
\end{array}\right.
\end{equation}

Choosing $\r, \q$ such that 
\[  \r(1-p) -4\q =0, \quad \r-\q+1=0, \]
namely 
\[\q= \frac{p-1}{p+3}, \quad \r = -\frac{4}{p+3},\]
we get $v'(0,\a)=1$ for every $\a>0$ and 
\begin{align}
\lim_{\a\to+ \infty} v'''(0,\a)& = \lim_{\a\to\infty} \a^{\r-3\q}\b(\a) =\lim_{\a\to +\infty} E \a^{\r-3\q-1} - \frac k 2\a^{\r-3\q+1}\\
&=\lim_{\a\to+\infty} E \a^{-2-2\q} - \frac k 2\a^{-2\q}= 0.
\end{align}

Now consider 
\[ g(t) = \lim_{\a\to + \infty} \a^{\r-4\q}f(\a^{-\r}t). \]
Then $g$ is well defined as \eqref{f2int} gives us
\[ \rho |t|^{p+1}\leq g(t)t\leq \bar\beta |t|^{p+1}, \ \forall \ t\in\mathbb{R},\]
for some $\bar\beta\geq \beta$. Moreover, $g: \R \to \R$ is Lipschitz. In particular $g$ satisfies \eqref{f1}.

Therefore, as $\a\to + \infty$, we see that the problem
\begin{equation}\label{reesc}
\left\{
\begin{array}{l}
v{''''}+k \a^{-2\q}v'' + \a^{\r-4\q}f(\a^{-\r}v) =0, \\ 
(v,v',v'',v''')(0)=(0, 1, 0, \a^{\r-3\q}\b(\a)).
\end{array}\right.
\end{equation}
is a regular perturbation of 
\begin{equation}\label{limit}
\left\{
\begin{array}{l}
V{''''}+g(V) =0, \\ 
(V,V',V'',V''')(0)=(0, 1, 0, 0), 
\end{array}\right.
\end{equation}
so that $(v,v',v'',v''')(s,\a)\to (V,V',V'',V''')(s)$ as $\a\to+\infty$, and the convergence is uniform on bounded intervals.

Notice that, on the right hand side of $t=0$, the solution $V$ of \eqref{limit} satisfies
\begin{equation}\label{propriedadesV} V{''''}(t), \ V'''(t), \ V''(t) \textrm{ are negative, as long as $V$ remains positive.}\end{equation}
Set $T= \sup\{t>0; V'>0 \textrm{ over } (0,t)\}$. Then by \cite[Theorems 2 and 4]{gazzolapavani2013}, we know that $0 < T < + \infty$ and hence $V'(T) =0$. Moreover, from \eqref{propriedadesV} we know that ${V''(T)<0}$ and ${V'''(T)<0}$.

As $\a\to +\infty$, we get $m(\a)\sim \a^{-\q}T$ and 
\[w(m(\a), \a) \sim\a^{-\r}V(T), \quad \phi(\a) = w'''(m(\a),\a) \sim\a^{2\q-\r}V'''(T).\]
Observing that $\r<0$ and $2\q-\r = \q+1>0$ leads to 
\[w(m(\a),\a)\to + \infty, \quad \phi(\a)<0,\]
for sufficiently large $\a$.
\end{proof}

\begin{remark}
We stress that the above lemma above holds for any given $E \in \R$.
\end{remark}

\begin{lemma}\label{apequeno} Assume that $f$ is a $\mathcal{C}^1(\mathbb{R})$ function.
\begin{enumerate}[i)]
\item 
If $f$ satisfies \eqref{hextra1}, \eqref{hextra2}, $E=0$ and $\a>0$ is small enough, then ${w'''(m(\a ),\a )>0}$.
\item If $f$ satisfies \eqref{hextra1}, $f'(0)=0$, $E=0$ and $\a>0$ is small enough, then $w'''(m(\a ),\a )>0$.
\end{enumerate}
\end{lemma}

\begin{proof}
In order to study the behavior of the problem as $\a \to 0$, let us rescale as above with $\q=0$ and $\r=-1$, in this case we arrive at 
\begin{equation} \label{alfazero}
\left \{ \begin{array}{l} 
v^{iv} + kv''+ \a^{-1} f( \a v) = 0, \\
(v,v',v'',v''')(0) =(0,1,0, \a^{-1} \b ( \a )).
\end{array}\right.
\end{equation}

Letting  $ \a \to 0$ from above, we get 
\[ 
\lim_{\a \to 0} \a^{-1} \b ( \a) = \lim_{ \a \to 0} \Bigg( E \a^{-2} - \frac{k}{2} \Bigg) = \frac{-k}{2}, \quad \mbox{if } E=0,
\]

\begin{align*}
\lim_{\a \to 0} \a^{-1} f( \a v) &= \lim_{\a \to 0} \a^{-1} (f'(0) \a v + o (\a v)) \\
&= \lim_{\a \to 0} f'(0)v + \a^{-1} o(\a v) = f'(0)v.
\end{align*}
We then get the limit problem
\begin{equation}\label{limitezero} 
\left \{ \begin{array}{l}
V{''''} + kV''+ f'(0)V = 0, \\
(V,V',V'',V''')(0) = (0,1,0, -k/2).
\end{array}\right.
\end{equation}  

\medbreak
\noindent \textsc{Proof of \textit{i)}}
Under \eqref{hextra2}, the solution of problem \eqref{limitezero} is given explicitly by
\[
V(t) = C_1 \sin (\lambda_1 t)+ C_2 \sin (\lambda_2 t),
\]
where $\lambda_1$ and $\lambda_2$ are the imaginary part of the roots of the characteristic equation,
\[
\lambda_1 = \sqrt{\frac{k + \sqrt{k^2 - 4 f'(0)}}{2}}, \quad \lambda_2 = \sqrt{\frac{k - \sqrt{k^2-4f'(0)}} 2},\]
both of which are positive. 
From the initial condition, we get
\[
C_1 = \frac 1{2\lambda_1}, \quad C_2 = \frac 1{2\lambda_2}.
\]

Let $T$ be the first critical point of $V$ on $(0, \infty)$, so that $V(T)>0$ and $V'(T)=0$, that is,
\begin{equation}\label{VlinhaT}
\cos (\lambda_2 T) = - \cos (\lambda_1 T).
\end{equation}
Evaluating $V'''(T)$, we get
\begin{align*}
V'''(T) &= - C_1 \lambda_1^3 \cos (\lambda_1 T) -  C_2 \lambda_2^3 \cos (\lambda_2 T)  = - \frac12\lambda_1^2 \cos (\lambda_1 T)-\frac12\lambda_2^2 \cos(\lambda_2 T)\\
 &= \frac12\lambda_1^2 \cos (\lambda_2 T)-\frac12\lambda_2^2 \cos(\lambda_2 T) = 
 \frac12\sqrt{k^2-4f'(0)}\cos(\lambda_2 T)>0.
\end{align*}
Therefore, for small enough $\a$,
\[
w'''(m(\a),\a) \sim\a V'''(T)>0.
\]

\medbreak
\noindent \textsc{Proof of \textit{ii)}}
In this case the limit problem is
\begin{equation}\label{limitezerodois} 
\left \{ \begin{array}{l}
V{''''} + kV''= 0, \\
(V,V',V'',V''')(0) = (0,1,0, -k/2).
\end{array}\right.
\end{equation}  
The solution is then given by 
\[
V(t) = C_1 \sin (\sqrt{k} t)+ C_2 t,
\]
and the initial condition yields
\[C_1 =\frac1{2\sqrt{k}} , \quad C_2 = \frac12.\]

For any $T>0$ such that $\cos(\sqrt{k}T)=-1$, 
straight calculation shows that $V'(T)=0$ and
\[V'''(T) = -\frac k 2 \cos(\sqrt{k}t) = \frac k2>0.\]
Taking $T$ to be the first of such values, the conclusion follows as in the previous case.
\end{proof}

\end{document}